\documentclass[11pt,a4paper,oneside]{article}
\usepackage[english]{babel}
\usepackage[T1]{fontenc}
\usepackage{amsmath}
\usepackage{amsthm}
\usepackage{amsfonts}
\usepackage{amssymb}
\usepackage{lmodern} 
\usepackage{indentfirst}		
\usepackage{microtype} 			
\usepackage[numbers]{natbib}
\usepackage{hyphenat}
\usepackage{pgfplots}  	
\usepackage[latin1]{inputenc}
\usepackage{indentfirst}

\definecolor{blackgreen}{RGB}{0,80,0}

\usepackage[plainpages=false,pdfpagelabels,pdfusetitle,pdfdisplaydoctitle, pdfstartpage=1, pdfstartview={{XYZ null null 1.00}}]{hyperref}
\hypersetup{
	pdftitle={},    
	pdfauthor={},    
	pdfnewwindow=true,      
	colorlinks=true,      
	linkcolor=blackgreen,         
	citecolor=red}

\usepackage[top=2cm, bottom=2cm, left=2cm, right=2cm]{geometry}
\usepackage{mathtools}
\mathtoolsset{showonlyrefs}

\newcommand{\N}{\mathbb{N}}

\newcommand{\R}{\mathbb{R}}

\newcommand{\T}{\mathbb{T}}

\newcommand{\Z}{\mathbb{Z}}

\newcommand{\supp}{\operatorname{supp}}
\numberwithin{equation}{section}

\newtheorem{thm}{Theorem}[section]

\newtheorem{lemma}[thm]{Lemma}

\newtheorem{proposition}{Proposition}[section]

\newtheorem{theorem}{Theorem}[section]

\newtheorem{remark}{Remark}[section]

\numberwithin{equation}{section}

\title{A note on the well-posedness in the energy space for the
generalized ZK equation posed on $\mathbb{R}\times\mathbb{T}$}
\author{ Luiz Gustavo Farah and Luc Molinet } 
\date{} 

\begin{document}
\maketitle
	
\begin{abstract}\noindent
In this note, we prove the local well-posedness in the energy space of the $k$-generalized Zakharov-Kuznetsov equation posed on $ \R\times \T $  for any power non-linearity $ k\ge 2$. Moreover, we obtain  global solutions under a precise smallness assumption on the initial data by proving a sharp Gagliardo Nirenberg type inequality.
\end{abstract}

\section{Introduction}
We consider the generalized Zakharov--Kuznetsov (gZK) equation with any power non-linearity
\begin{equation}\label{ZKk}
u_t +\partial_x ( u^{k+1}) + u_{xxx} + u_{xyy} = 0, \quad (x,y)\in \R\times \T,
\end{equation}
where $ u=u(t,x,y) $ is a real-valued function, $ (t,x,y)\in \R\times\R\times \T $ with  $\T=\R/\Z$,  and where $k\geq 2$.  In the case $ k=1$, this equation was derived by Zakharov and Kuznetsov in \cite{ZK74} to describe the propagation of ionic-acoustic waves in magnetized plasma when  weak transverse effects occur in dimension three. The derivation of ZK from the Euler-Poisson system with magnetic field was performed by Lannes, Linares and Saut \cite{LLS13}.  More generally, the  gZK is a natural two-dimensional extension of the generalized Korteweg-de Vries (gKdV) equation in the case of nearly one dimensional propagation. 

The gZK equation has an Hamiltonian structure. In particular, it is well-known that the two
following quantities are conserved for smooth solutions
\begin{equation} \label{M}
M(u)=\int u ^2dxdy,
\end{equation}
and
\begin{equation} \label{H}
H(u)= \int \left( \frac12|\nabla u|^2-\frac1{k+2}u^{k+2}\right)dxdy.
\end{equation}

A particular solution of the form $u(t,x,y)=u_c(x-ct,y)$ is called a traveling wave with speed $c$. For the gZK equation, we are interested in a non-trivial solution $u_c$ to the following stationary equation,
\begin{equation}
\label{TW}
- \partial_{xx} u_c - \partial_{yy} u_c + c u_c
- u_c^{3} = 0, \quad (x,y)\in \R\times \T. 
\end{equation}
The question of stability and instability of this traveling waves is quite interesting. In the $\mathbb{R}^2$ case, the orbital stability was first studied by de Bouard \cite{Bou96} showing that positive and radially symmetric traveling waves are orbitally stable for $k < 2$ and unstable for $k > 2$. More recently, the critical case $k = 2$ was considered by Holmer, Roudenko and the second author \cite{FHR19} showing the orbital instability (see also \cite{FHR19II} for an alternative proof of instability in the supercritical gZK case $k>2$). The more refined asymptotic stability was obtained by C\^ote, Mu\~noz, Pilod and Simpson \cite{CMP16} for $1\leq k<k_0\approx 1.15$, adapting the ideas introduced by Martel and Merle \cite{MM01}, \cite{MM05} and \cite{MM08} for the gKdV equation to a multidimensional model. The special case where the traveling wave is independent of the second variable, that is $u(t,x,y)=Q_c(x-ct)$, is commonly refereed as a line soliton of \eqref{ZKk} and, in this case, $Q_c$ is the soliton of speed $c > 0$ to the gKdV equation. A natural question is the stability of such solutions with respect to perturbations which are periodic in the transversal direction. In the $\R^2$ case with $k=1$, the transverse non-linear instability for line soliton was obtained by Rousset and Tzvetkov \cite{RT08}. In the $\mathbb{R} \times \mathbb{T}$ case with $k=1$, the linear instability was showed by Bridges \cite{Br00} for sufficiently large speeds. Later, Yamazaki \cite{YY17} studied the non-linear orbital stability and  transverse instability applying the Evans' function method by Pego and Weinstein \cite{PW92} and the ideas introduced by Rousset and Tzvetkov \cite{RT09} to deduce a full description according to the speed of the line soliton. 

When posed on  $\R^2$, the local wel-posedness of the gZK equation for has been studied in many papers. Let us quote \cite{BL01}, \cite{Fa95}, \cite{FLP12}, \cite{G15}, \cite{GH14}, \cite{LP09}, \cite{LP11}, \cite{RV12} and more recently Kinoshita \cite{K21} and \cite{K22} where local well-posedness in $ H^{-1/4+}(\R^2) $ for $k=1$ and $ H^{1/4}(\R^2) $ for $k=2$ are proved. On the other hand, the case $ \R\times\T $ has been less studied. Let us mention the work of Pilod and the first author \cite{MP15} where the well-posedness of the ZK equation in the energy space $ H^1(\R\times\T) $ for $k=1$ was proved. Our main goal here is to address the local and global theory in $\R\times \T$ for $k\geq 2$.


Our first main result in this work is the local well-posedness in the Sobolev type spaces $H^s(\mathbb R\times \T)$,  in particular in the energy space $ H^1(\R\times\T) $, of the gZK equation \eqref{ZKk}. 
 To do  this we will perform a fixed point theorem in Bourgain's type spaces associated to the linear ZK equation and establish a new $L^4$-Strichartz estimate in this context.

\begin{theorem} \label{LWP}
There exists a function $ s_0 : \Z_+ \setminus\{0,1\} \to [5/6,1[ $, with 
 $s_0(2)=5/6 $, such that for  any $ k\ge 2 $ , any $s>s_0(k)$ 
 and  any $u_0 \in H^s(\mathbb R\times \T)$, there exists $T=T(\|u_0\|_{H^{s_0(k)+}})>0$ and a unique solution of \eqref{ZKk} such that $u(\cdot,0)=u_0$ and 
\begin{equation} \label{theoR2.1}
u \in C([0,T]:H^s(\mathbb R\times\T)) \cap X_T^{s,\frac12+} \ .
\end{equation}
Moreover, for any $T' \in (0,T)$, there exists a neighborhood $\mathcal{U}$ of $u_0$ in $H^s(\mathbb R\times\T)$, such that the flow map data-solution 
\begin{equation} \label{theoR2.2}
S: v_0 \in \mathcal{U} \mapsto v \in  C([0,T']:H^s(\mathbb R\times\T)) \cap X_{T'}^{s,\frac12+}
\end{equation}
is smooth.
\end{theorem}
Another very interesting question is to derive an optimal bound to get the global existence of solutions. Such bound has been now derived by \cite{KM06}, \cite{HR08}, \cite{DHR08}, \cite{FLP11}, \cite{FLP12}, \citep{FLP14} for several dispersive equations posed on $\R^n $ or $ \R $ as the non-linear Schr\"odinger equation, the generalized KdV equations, the generalized ZK equation and the $k$-dispersion generalized Benjamin-Ono equation. In all these works the optimal bound is obtained by finding the solutions of a variational problem (see, for instance, Weinstein \cite{W83} and Angulo, Bona, Linares and Scialom \cite{ABLS02}). However, the problem on $ \R\times\T $ seems to be more involved, since we need a new sharp Gagliardo-Nirenberg inequality adapted to this setting. The first step in this direction is the following estimate in the spirit of Hebey and Vaugon \cite{HV95} (we refer the reader to the works of  Yu, Yue and Zhao \cite{YYZ21} and Luo \cite{LY21} for related results). 
\begin{lemma}\label{LemGN}
Let $f\in H^1(\mathbb{R}\times\mathbb{T})$ and $k\geq 1$, then there exists a universal constant $C_{k,\mathbb{T}}>0$ such that
\begin{equation}\label{SGN}
\|f\|^{k+2}_{k+2}\leq C_{k,\mathbb{R}} \|f\|^2_2\left(\|\nabla f\|^2_2+C_{k,\mathbb{T}}\|f\|^2_2\right)^{k/2},
\end{equation}
where $C_{k,\mathbb{R}}$ is the best constant for the inequality
\begin{equation}\label{GN}
\|V\|^{k+2}_{k+2}\leq C_{k,\mathbb{R}} \|V\|^2_2 \|\nabla V\|_2^{k},
\end{equation}
with $V\in H^1(\mathbb{R}^2)$.
\end{lemma}

\begin{remark}\label{Rema1}
Note that
\begin{itemize}
\item The inequality \eqref{SGN} cannot hold with $C_{k,\mathbb{T}}=0$. Indeed, if this is the case, then taking $f(x,y)=g(x)\in H^1(\mathbb{R})$ we obtain
$$
\|g\|^{k+2}_{L^{k+2}(\mathbb{R})}\lesssim  \|g\|^2_{L^2(\mathbb{R})}\|g'\|^k_{L^2(\mathbb{R})},
$$
which, by a scaling argument, is not true for all $g(x)\in H^1(\mathbb{R})$.
\item The constant $C_{k,\mathbb{R}} $ is sharp in the sense that it cannot be replaced by a smaller number such that \eqref{SGN} holds for any positive constant $c>0$ inside the parentheses in the right hand side. Indeed, assume that there exists $D_k<C_{k,\mathbb{R}}$ such that for some positive constant $c>0$ it holds
\begin{equation}\label{g1}
\|f\|^{k+2}_{k+2}\leq  D_k \|f\|^2_2\left(\|\nabla f\|^2_2+c\|f\|^2_2\right)^{k/2}, \,\,\,\mbox{for all}\,\,\, f\in H^1(\mathbb{R}\times\mathbb{T}).
\end{equation}
Let $f_{\delta}$ be a smooth function compactly supported in an open ball $B_{\delta}\subset \mathbb{R}\times [0,1]$ for some radius $\delta>0$ to be chosen later. From the Holder inequality we have
\begin{equation}\label{g2}
\|f_{\delta}\|_{2}\leq |B_{\delta}|^{\frac{k}{2(k+2)}}\|f_{\delta}\|_{k+2}.
\end{equation}

Taking $\varepsilon>0$ such that $D_k(1+\varepsilon)<C_{k,\mathbb{R}}$ and then $c_{k,\varepsilon}>0$ such that
$$
(a+b)^{k/2}\leq (1+\varepsilon)a^{k/2}+c_{k,\varepsilon}b^{k/2}, \,\,\,\mbox{for all}\,\,\, a,b>0,
$$
it follows from \eqref{g1}-\eqref{g2} that 
$$
\|f_{\delta}\|^{k+2}_{k+2}\leq  D_k (1+\varepsilon)\|f_{\delta}\|^2_2\|\nabla f_{\delta}\|^k_2+cD_kc_{k,\varepsilon}|B_{\delta}|^{k/2}\|f_{\delta}\|^{k+2}_{k+2}.
$$
So, for $\delta>0$ small enough, we can find $D_k<D'_k<C_{k,\mathbb{R}}$ such that
$$
\|f_{\delta}\|^{k+2}_{k+2}\leq  D'_k \|f_{\delta}\|^2_2\|\nabla f_{\delta}\|^k_2,
$$
reaching a contradiction with the sharp Gagliardo-Nirenberg \eqref{GN}.
\end{itemize}
\end{remark}



From the celebrated work of Weinstein \cite{W83}, the best constant $C_{k,\mathbb{R}}>0$ for the
Gagliardo-Nirenberg inequality in $\mathbb{R}^2$ \eqref{GN} is closely related to a particular solution of
\begin{equation}\label{ground}
\Delta Q-Q+Q^{k+1}=0.
\end{equation}
Indeed, denoting by $Q_k$ the unique positive radial solution of \eqref{ground}, it is known that
\begin{equation}\label{opt1}
C_{k,\mathbb{R}}=\frac{2^{\frac{k-2}{2}}(k+2)}{k^{\frac k 2}\|Q_k\|_{2}^k}.
\end{equation}
Moreover, we have the following identities
$$
\frac{ k}{ 2} \|Q_k\|_{{2}}^{2}=\|\nabla Q_k\|_{{2}}^{2} \quad \mbox{and} \quad \|Q_k\|^{k+2}_{k+2}=\frac{k+2}{2}\|Q_k\|_{{2}}^{2} .
$$
From the definition of the conserved quantities mass and energy \eqref{M}-\eqref{H}, we also have, for $s_k=(k-2)/k$, that
\begin{equation}\label{rel2}
H(Q_k)^{s_k} M(Q_k)^{1-s_k}=\left(\frac{k-2}{4}\right)^{\frac{k-2}{k}}\|Q_k\|_{{2}}^{2}\quad \mbox{and} \quad \|\nabla
Q_k\|_{2}^{s_k}\|Q_k\|_{2}^{1-s_k}=\left(\frac{k}{2}\right)^{\frac{k-2}{2k}}\|Q_k\|_{{2}}.
\end{equation}

We now state our global well-posedness theory for small data in the energy space (a similar result in the $\mathbb{R}^2$ case can be found in \cite[Theorem 1.3]{FLP12}).
\begin{theorem}\label{global1}
Let $ k\ge 2$. 
We denote by $Q_k$ the unique positive radial solution of equation \eqref{ground} and 
 for $u_0\in H^1(\mathbb{R}\times\mathbb{T})$  we called $u\in C([-T,T]; H^1(\R\times\T) $, with $T>0 $, the corresponding solution to the gZK equation.
\begin{itemize}
\item[(i)] Assume $k=2$. If 
\begin{equation}\label{GR0}
\|u_0\|_{2}< \|Q_2\|_{2},
\end{equation}
then the solution $u$ exists globally in time, i.e. $u\in C_b(\R;H^1(\R\times\T))$.
\item[(ii)] Assume $k\geq 3$ and let $s_k=(k-2)/k$. If
\begin{equation}\label{GR1}
\left(H(u_0)+\frac{C_{k,\mathbb{T}}}{2}M(u_0)\right)^{s_k} M(u_0)^{1-s_k} < H(Q_k)^{s_k} M(Q_k)^{1-s_k}.
\end{equation}
and 
\begin{equation}\label{GR2}
\left(\|\nabla u_0\|^2_{2}+C_{k,\mathbb{T}}\|u_0\|_{2}^2\right)^{\frac{s_k}{2}}\|u_0\|_{2}^{1-s_k} < \|\nabla
Q_k\|_{2}^{s_k}\|Q_k\|_{2}^{1-s_k},
\end{equation}
then the solution $u$ exists globally in time, i.e. $u\in C_b(\R;H^1(\R\times\T))$.
\end{itemize}
\end{theorem}

\begin{remark}
Before we proceed let us make three observations.
\begin{itemize}
\item The restriction in $(i)$ is the same as in the $\mathbb{R}^2$ case (see, for instance, \cite[Theorem 1.3]{LP09} and \cite[Remark 1.4]{FLP11}).
\item Taking $k=2$ in $(ii)$ we recover $(i)$.
\item For $k\geq 3$, the assumption \eqref{GR2} implies that $H(u_0)+\frac{C_{k,\mathbb{T}}}{2}M(u_0)$ is a non negative quantity. Indeed, by the sharp Gagliardo Nirenberg type inequality \eqref{SGN} and identities \eqref{opt1} and  \eqref{rel2}, we obtain
\begin{eqnarray}
H(u_0)+\frac{C_{k,\mathbb{T}}}{2}M(u_0)&\geq & \frac{1}{2}\left(\|\nabla u_0\|^2_{2}+C_{k,\mathbb{T}}\|u_0\|_{2}^2\right)-\frac{C_{k,\mathbb{R}}}{k+2}\|u_0\|_{2}^2\left(\|\nabla u_0\|^2_2+C_{k,\mathbb{T}}\|u_0\|^2_2\right)^{k/2}\\
&=&\frac{1}{2}\left(\|\nabla u_0\|^2_{2}+C_{k,\mathbb{T}}\|u_0\|_{2}^2\right)\left(1-\frac{2C_{k,\mathbb{R}}}{k+2}\left((\|\nabla u_0\|^2_2+C_{k,\mathbb{T}}\|u_0\|^2_2)^{\frac{
s_k}{2}}\|u_0\|_{2}^{1-s_k}\right)^k\right)\\
&\geq &\frac{1}{2}\left(\|\nabla u_0\|^2_{2}+C_{k,\mathbb{T}}\|u_0\|_{2}^2\right)\left(1-\frac{2C_{k,\mathbb{R}}}{k+2}\left(\|\nabla
Q_k\|_{2}^{s_k}\|Q_k\|_{2}^{1-s_k}\right)^k\right)\\
&=&\frac{k-2}{2k}\left(\|\nabla u_0\|^2_{2}+C_{k,\mathbb{T}}\|u_0\|_{2}^2\right).
\end{eqnarray}
\end{itemize}
\end{remark}
%
\begin{remark}
 Theorem \ref{LWP} also ensures the global existence of $ H^1(\R\times\T) $-solutions for the defocusing version of \eqref{ZKk}, i.e. 
$$
u_t -\partial_x ( u^{k+1}) + u_{xxx} + u_{xyy} = 0,
$$
with $ k $ odd. Indeed, in this case the energy conservation law (see \eqref{H}) has a positive sign in front of the potential term which is positive since $ k $ is odd. Therefore, together with the mass conservation (see \eqref{M}), the $ H^1(\R\times\T) $ norm of the solution is uniformly bounded in time and, since the time of existence obtained in Theorem \ref{LWP} depends on the norm of the initial data, the solution is global.
\end{remark}
The plan of this paper is the following: in the next section we introduce our notations and the function spaces we will work with. In Section 3, we  establish our $ L^4$-Strichartz estimate and prove the local well-posedness result. Finally, the last section is devoted to the proof of Lemma \ref{LemGN} followed by the global well-posedness result.

\section{Preliminaries}

In this section we introduce some notation and also the function spaces used  in the sequel.

\subsection{Notations}
For any positive numbers $a$ and $b$, the notation $a \lesssim b$
means that there exists a positive constant $c$ such that $a \le c
b$. We also write $a \sim b$ when $a \lesssim b$ and $b \lesssim a$.
If $\alpha \in \mathbb R$, then $\alpha+$, respectively $\alpha-$,
will denote a number slightly greater, respectively lesser, than
$\alpha$. If $A$ and $B$ are two positive numbers, we use the
notation $A\wedge B=\min(A,B)$ and $A \vee B=\max(A,B)$. Finally,
$\text{mes} \, S$ or $|S|$ denotes the  measure of a
measurable set $S$ of $\R $, $ \T $ or $ \R\times\Z $, endowed with their usual measures.

We use the notation $|(\xi,q)|=\sqrt{3\xi^2+q^2}$ for  $(\xi,q) \in \mathbb R\times \Z$.
For $u=u(t,x,y) \in \mathcal{S}(\mathbb R^2\times\T)$, $\mathcal{F}(u)$, or
$\widehat{u}$, will denote its space-time Fourier transform, whereas
$\mathcal{F}_{xy}(u)$, or $(u)^{\wedge_{xy}}$, respectively
$\mathcal{F}_t(u)=(u)^{\wedge_t}$, will denote its Fourier transform
in space, respectively in time.

Throughout the paper, we fix a smooth cutoff function $\eta$ such that
\begin{displaymath}
\eta \in C_0^{\infty}(\mathbb R), \quad 0 \le \eta \le 1, \quad
\eta_{|_{[-5/4,5/4]}}=1 \quad \mbox{and} \quad  \mbox{supp}(\eta)
\subset [-8/5,8/5].
\end{displaymath}
For $k \in \mathbb N^{\star}=\mathbb Z \cap [1,+\infty)$, we define
\begin{displaymath}
\phi(\xi)=\eta(\xi)-\eta(2\xi), \quad
\phi_{2^k}(\xi,q):=\phi(2^{-k}|(\xi,q)|).
\end{displaymath}
and
\begin{displaymath}
\psi_{2^k}(\xi,q ,\tau)=\phi(2^{-k}(\tau-(\xi^3+\xi q^2))).
\end{displaymath}
By convention, we also denote
\begin{displaymath}
\phi_1(\xi,q)=\eta(|(\xi,q)|), \quad \text{and} \quad
\psi_1(\xi,q,\tau)=\eta(\tau-(\xi^3+\xi q^2)).
\end{displaymath}
Any summations over capitalized variables such as $N, \, L$, $K$ or
$M$ are presumed to be dyadic with $N, \, L$, $K$ or $M \ge 1$,
\textit{i.e.}, these variables range over numbers of the form $\{2^k
: k \in \mathbb N\} $ where $ \N= \Z\cap \R_+ $. Then, we have that
\begin{displaymath}
\sum_{N}\phi_N(\xi,\mu)=1, \quad \mbox{supp} \, (\phi_N) \subset
\{\frac{5}{8}N\le |(\xi,\mu)| \le \frac{8}5N\}=:I_N \ , \ N \ge 2,
\end{displaymath}
and
\begin{displaymath}
\mbox{supp} \, (\phi_1) \subset \{|(\xi,q)| \le \frac85\}=:I_1  .
\end{displaymath}
Let us define the Littlewood-Paley multipliers by
\begin{equation}\label{proj}
P_Nu=\mathcal{F}^{-1}_{xy}\big(\phi_N\mathcal{F}_{xy}(u)\big), \quad
Q_Lu=\mathcal{F}^{-1}\big(\psi_L\mathcal{F}(u)\big).
\end{equation}

Finally, we denote by $e^{-t\partial_x\Delta}$ the free group
associated with the linearized part of equation \eqref{ZKk}, which is
to say,
\begin{equation} \label{V}
\mathcal{F}_{xy}\big(e^{-t\partial_x\Delta}\varphi
\big)(\xi,q)=e^{itw(\xi,q)}\mathcal{F}_{xy}(\varphi)(\xi,q),
\end{equation}
where $w(\xi,q)=\xi^3+\xi q^2$. We also define the resonance
function $H$ by
\begin{equation} \label{Resonance}
H(\xi_1,q_1,\xi_2,q_2)=w(\xi_1+\xi_2,q_1+q_2)-w(\xi_1,q_1)-w(\xi_2,q_2).
\end{equation}
Straightforward computations give that
\begin{equation} \label{Resonance2}
H(\xi_1,q_1,\xi_2,q_2)=3\xi_1\xi_2(\xi_1+\xi_2)+\xi_2q_1^2+\xi_1q_2^2+2(\xi_1+\xi_2)q_1q_2.
\end{equation}

\subsection{Function spaces}
For $1 \le p \le \infty$, $L^p(\mathbb R\times\T)$ is the usual Lebesgue
space with the norm $\|\cdot\|_{L^p}$, and for $s \in \mathbb R$ ,
the real-valued Sobolev space $H^s(\mathbb R\times\T)$  denotes the space
of all real-valued functions with the usual norm
$\|u\|_{H^s}=\|\sqrt{1+\xi^2+q^2} \; \hat{u}\|_{L^2(\R\times\Z)}.$
If $u=u(x,y,t)$ is a function defined for $(x,y) \in
\mathbb R\times\T$ and $t$ in the time interval $[0,T]$, with $T>0$, and $B$
is one of the spaces defined above, $1 \le p \le \infty$ and $1 \le q \le \infty$, we will
define the mixed space-time spaces $L^p_TB_{xy}$ and 
$L^p_tB_{xy}$ by the norms
\begin{displaymath}
\|u\|_{L^p_TB_{xy}} =\Big(
\int_0^T\|u(\cdot,\cdot,t)\|_{B}^pdt\Big)^{\frac1p} \quad , \quad
\|u\|_{L^p_tB_{xy}} =\Big( \int_{\mathbb R}\|u(\cdot,\cdot,t)\|_{B}^pdt\Big)^{\frac1p},
\end{displaymath}
if $1 \le p, \ q < \infty$ with the obvious modifications in the case $p=+\infty$ or $q=+\infty$.

For $s$, $b \in \mathbb R$, we introduce the Bourgain spaces
$X^{s,b}$ related to the linear part of \eqref{ZKk} as
the completion of the Schwartz space $\mathcal{S}(\mathbb R^3)$
under the norm
\begin{equation} \label{Bourgain}
\|u\|_{X^{s,b}} = \left(
\sum_{q\in\Z} \int_{\mathbb{R}^2}\langle\tau-w(\xi,q)\rangle^{2b}\langle
|(\xi,q)|\rangle^{2s}|\widehat{u}(\xi,q, \tau)|^2 d\xi 
d\tau \right)^{\frac12},
\end{equation}
where $\langle x\rangle:=\sqrt{1+|x|^2}$. Moreover, we define a localized (in time) version of these spaces.
Let $T>0$ be a positive time. Then, if $u: \mathbb R\times\T \times
[0,T]\rightarrow \mathbb R$, we have that
\begin{displaymath}
\|u\|_{X^{s,b}_{T}}=\inf \{\|\tilde{u}\|_{X^{s,b}} \ : \ \tilde{u}:
\mathbb R\times\T \times \mathbb R \rightarrow \mathbb R, \
\tilde{u}|_{\mathbb R\times\T \times [0,T]} = u\}.
\end{displaymath}

\section{Local well-posedness theory}
Our main tools to prove the well-posedness  result are a  $L^4$-Strichartz estimate that seems to be new and a bilinear estimate that already appeared in \cite{MP15}.
\subsection{ Strichartz and bilinear estimates}
We will need the following technical lemmas (see for instance  Saut and Tzvetkov \cite{ST01}):
\begin{lemma} \label{basic1}
Consider a set $\Lambda \subset X \times Y$, where
$X=\mathbb R$ or $\mathbb Z$ and $  Y \in \{\R,\Z,\R\times\Z \}$. Assume that $ \Lambda \subset I \times Y $ with  $I \subset \mathbb R$ and  that
there exists $C>0$ such that for any fixed $x_0 \in  I \cap X$, $|\Lambda
\cap \{(x_0,y), \, y  \in Y\}| \le C$. Then, we get that 
$|\Lambda| \le C |I|$ in the case where $X=\mathbb R$ and $|\Lambda| \le C( |I|+1)$ 
in the case where $X=\mathbb Z$.
\end{lemma}

\begin{lemma} \label{basic2}
Let $a \neq 0, \ b, \ c$ be real numbers and $I$ be an interval on
the real line. Then,
\begin{equation} \label{basicIII1}
\text{mes} \, \{x \in \R \ / \ ax^2+bx+c \in I\} \lesssim
\frac{|I|^{\frac12}}{|a|^{\frac12}}.
\end{equation}
and
\begin{equation} \label{basicIII2}
\# \{q \in \mathbb Z  \ / \ aq^2+bq+c \in I\} \lesssim
\frac{|I|^{\frac12}}{|a|^{\frac12}}+1.
\end{equation}
\end{lemma}
Let us  now prove a Strichartz estimate  in the context of Bourgain's spaces in $ \R\times\T $.
\begin{proposition}[$L^4$-Strichartz estimate] \label{prop31} For any $ u\in X^{1/6,3/8}(\R\times \T) $ it holds 
\begin{equation}\label{stri4}
\| u \|_{L^4_{txy}} \lesssim \|u\|_{X^{1/6,3/8}} \; .
\end{equation}
\end{proposition}
\begin{proof}
\eqref{stri4} is a direct consequence of the following bilinear estimate 
\begin{equation}\label{bi2}
\|P_{N_1} Q_{L_1} u_1  P_{N_2} Q_{L_2} u_2 \|_{L^2_{txy}} \lesssim 
(N_1\wedge N_2)^{1/3} (L_1\wedge L_2)^{1/2} (L_1\vee L_2)^{1/4} \|P_{N_1}Q_{L_1}u\|_{L^2}\,\|P_{N_2}Q_{L_2}u_2\|_{L^2}
\end{equation}
where $u_1$ and $u_2$ are two functions in $L^2(\R^2\times \T)$.
Indeed with \eqref{bi2} in hand, we get the following chain of inequalities
\begin{align*}
\| u^2\|_{L^2} \le & \; 2  \sum_{N_1\ge N_2} \sum_{L_1\ge L_2} 
\|P_{N_1} Q_{L_1} u  P_{N_2} Q_{L_2} u \|_{L^2_{txy}}
+ \|P_{N_1} Q_{L_2} u_1  P_{N_2} Q_{L_1} u \|_{L^2_{txy}}\\
 \lesssim &  \sum_{N_1\ge N_2} \sum_{L_1\ge L_2}  N_2^{1/3} L_1^{1/4} L_2^{1/2} \Bigl( 
\|P_{N_1} Q_{L_1} u\|_{L^2}  \|P_{N_2} Q_{L_2} u \|_{L^2}+ \|P_{N_1} Q_{L_2} u\|_{L^2}  \|P_{N_2} Q_{L_1} u \|_{L^2}\Bigr) \\
\lesssim &  \sum_{m,m_2\ge 0} \sum_{l,l_2\ge 0}  2^{m_2/3} 2^{(l_2+l)/4} 
2^{l_2/2} \\
 & \hspace*{20mm} \Bigl( 
\|P_{2^{(m_2+m)}} Q_{2^{(l_2+l)}} u\|_{L^2}  \|P_{2^{m_2}} Q_{2^{l_2}} u \|_{L^2}
 + \|P_{2^{(m_2+m)}} Q_{2^{l_2}} u\|_{L^2}  \|P_{2^{m_2}} Q_{2^{(l_2+l)}} u \|_{L^2}\Bigr) \\
 \lesssim &  \sum_{m,m_2\ge 0} \sum_{l,l_2\ge 0}  2^{m_2/6} 2^{(m_2+m)/6} 2^{-m/6} 2^{3(l_2+l)/8} 
2^{3 l_2/8} 2^{-l/8} \\
 & \hspace*{10mm}\Bigl( 
\|P_{2^{(m_2+m)}} Q_{2^{(l_2+l)}} u\|_{L^2}  \|P_{2^{m_2}} Q_{2^{l_2}} u \|_{L^2}+ \|P_{2^{(m_2+m)}} Q_{2^{l_2}} u\|_{L^2}  \|P_{2^{m_2}} Q_{2^{(l_2+l)}} u \|_{L^2}\Bigr) \\
 \lesssim &  \sum_{m,l \ge 0} 2^{-\frac{m}{6}} 2^{-\frac{l}{8}} \Bigl[ 
\Bigl( \sum_{m_2,l_2\ge 0}  2^{\frac{m_2}{3}} 2^{\frac{3 l_2}{4} }  \|P_{2^{m_2}} Q_{2^{l_2}} u \|_{L^2}^2\Bigr)^{1/2}
 \Bigl( \sum_{m_2,l_2\ge 0}2^{\frac{m_2+m}{3}}2^{\frac{3(l_2+l)}{4}}
\|P_{2^{(m_2+m)}} Q_{2^{(l_2+l)}} u\|_{L^2}^2\Bigr)^{1/2}  \\
 & \hspace*{10mm}+\Bigl( \sum_{m_2,l_2\ge 0}  2^{\frac{(m_2+m)}{3}} 2^{\frac{3 l_2}{4} }  \|P_{2^{(m_2+m)}} Q_{2^{l_2}} u \|_{L^2}^2\Bigr)^{1/2}
 \Bigl(  \sum_{m_2,l_2\ge 0}2^{\frac{m_2}{3}}2^{\frac{3(l_2+l)}{4}}
\|P_{2^{m_2}} Q_{2^{(l_2+l)}} u\|_{L^2}^2\Bigr)^{1/2}\Bigr]   \\
  \lesssim & \| u\|_{X^{1/6,3/8}}^2
\end{align*}
Now to prove \eqref{bi2} we may first assume that $ u_1 $ and $ u_2$ have non negative real  space-time Fourier transform since 
$$
\|P_{N_1} Q_{L_1} u_1  P_{N_2} Q_{L_2} u_2 \|_{L^2_{txy}}=\|\phi_{N_1} \psi_{L_1} \widehat{u}_1 \ast 
 \phi_{N_2} \psi_{L_2} \widehat{u}_2\|_{L_2} \lesssim \|\phi_{N_1} \psi_{L_1} |\widehat{u}_1| \ast 
 \phi_{N_2} \psi_{L_2} |\widehat{u}_2|\|_{L_2} \;.
 $$ 
 and $ \| \phi_{N_1} \psi_{L_1} |\widehat{u_i}|\|_{L^2_{txy} }=\|P_{N_1} Q_{L_1} u_i \|_{L^2_{txy}}$, $i=1,2$. 
 Then it is crucial to notice that since $ u $ is real-valued, according to  the argument given by Bourgain \cite{B93II}  (see also the proof of Lemma 4 in \cite{ST01}), we may assume that $\supp  \widehat{u}_1 $ and $ \supp \widehat{u}_2  \subset \R\times \R_+\times \Z $. 
  By using the Plancherel theorem and the Cauchy-Schwarz inequality, we get
  \begin{equation}\label{la1}
  \|P_{N_1} Q_{L_1} u_1  P_{N_2} Q_{L_2} u_2 \|_{L^2_{txy}}
  \lesssim\sup_{(\tau,\xi,q)\in\R\times\R_+\times\Z}|A(\tau,\xi,q)|^{1/2} \|P_{N_1} Q_{L_1} u_1\|_{L^2_{txy}} \|P_{N_2} Q_{L_2} u_2\|_{L^2_{txy}}
 \end{equation}
 where, according to Lemma \ref{basic1},
\begin{eqnarray*}
|A(\tau,\xi,q)| 
 & = &  mes \Bigl\{ (\tau_1,\xi_1,q_1)\in \R\times  \R_+\times \Z  \, \slash \,  \; \xi-\xi_1\in  \R_+,
   \langle  \tau_1-\xi_1^3-\xi_1 q_1^2)\rangle \sim L_1, \\
  & & 
\langle \tau-\tau_1-(\xi-\xi_1)^2-(\xi-\xi_1)(q-q_1)^2 \rangle \sim L_2 , \; \langle (\xi_1,q_1) \rangle \sim N_1  \mbox{ and }\,\langle (\xi-\xi_1,q-q_1) \rangle \sim N_2 \Bigr\} \\
 & \lesssim &  (L_1\wedge L_2) \;  \text{mes} \,   B(\tau,\xi,q)\; , 
  \end{eqnarray*}
with 
\begin{eqnarray*}
   B(\tau,\xi,q) &:=& \Bigl\{  (\xi_1,q_1)\in \R_+\times \Z \ / \; \xi-\xi_1\ge 0,   \langle (\xi_1,q_1) \rangle \sim N_1 ,\,\langle (\xi-\xi_1,q-q_1) \rangle \sim N_2   \\
    & & \mbox{ and } \langle H(\xi_1,q_1,\xi-\xi_1,q-q_1) \rangle \lesssim L_1 \vee L_2  \Bigr\}  \; .
    \end{eqnarray*}
Here $ H $ is the resonance function defined in \eqref{Resonance}.
Notice that, in view of \eqref{Resonance2}, $  \xi_1 \mapsto H (\xi_1,q_1,\xi-\xi_1,q-q_1) $  and 
 $q_1 \mapsto H (\xi_1,q_1,\xi-\xi_1,q-q_1) $ 
 are  polynomial functions of order 2 with 
\begin{equation}\label{der2}
\partial^2_{\xi_1} H (\xi_1,q_1,\xi-\xi_1,q-q_1)= 3\xi \quad \text{and} \quad
\partial^2_{q_1} H (\xi_1,q_1,\xi-\xi_1,q-q_1)= 2\xi \; .
\end{equation}
We separate two cases : $ \xi\le (N_1\wedge N_2)^{2/3} $ and $ \xi\ge (N_1\wedge N_2)^{2/3} $.
On one hand, for $ \xi\le (N_1\wedge N_2)^{2/3} $, the second identity in \eqref{der2}  together with Lemmas \ref{basic1}-\ref{basic2} lead to 
$$
|B(\tau,\xi,q)| \lesssim \xi \Bigl( \Bigl( \frac{L_1\vee L_2}{2 \xi}\Bigr)^{1/2} +1\Bigr)
\lesssim (L_1\vee L_2)^{1/2} \xi^{1/2} + \xi \lesssim  (L_1\vee L_2)^{1/2} 
 (N_1\wedge N_2)^{2/3} 
$$
since $ 0\le \xi_1\le \xi \le (N_1\wedge N_2)^{2/3} $ 
 for all $(\tau_1,	\xi_1,q_1)\in A(\tau,\xi,q)$ and $N_1\wedge N_2 \ge 1$. On the other hand, for $ \xi\ge (N_1\wedge N_2)^{2/3} $, the first identity in \eqref{der2}  together with Lemmas \ref{basic1}-\ref{basic2} lead to 
$$
|B(\tau,\xi,q)| \lesssim (N_1\wedge N_2) \Bigl( \frac{L_1\vee L_2}{3 \xi}\Bigr)^{1/2} 
\lesssim  (L_1\vee L_2)^{1/2} 
 (N_1\wedge N_2)^{2/3} \;.
$$
Gathering the above estimates and \eqref{la1}, \eqref{bi2} follows that completes the proof of the proposition.
\end{proof}
We will also  make use of the following bilinear estimate that appeared in \cite{MP15}. 
\begin{proposition}[Bilinear  estimate]\label{prop32}
Let $N_1$, $N_2$,  be dyadic numbers in $\{2^k, k\in \N\}$ and let  $u_1$ and $u_2$  be two functions in $L^2(\R^3)$. Then, 
\begin{equation}\label{bi1}
\|(P_{N_1} u_1)(P_{N_2} u_2)\|_{L^2}
\lesssim \frac{(N_1 \wedge N_2)^{\frac{1}{2}+}}{(N_1\lor N_2)^{1-}}
\|P_{N_1} u_1\|_{X^{0,\frac{1}{2}-} }\,\|P_{N_2}u_2\|_{X^{0,\frac{1}{2}-}},
\end{equation}
whenever $N_1\ge 4N_2$ or $N_2\ge 4N_1$. 
\end{proposition}
\begin{proof} 
The trivial  estimate 
$$
\|P_{N_1} Q_{L_1} u_1  P_{N_2} Q_{L_2} u_2 \|_{L^2_{txy}} \lesssim 
(N_1\wedge N_2)(L_1\wedge L_2)^{1/2} \|P_{N_1}Q_{L_1}u\|_{L^2}\,\|P_{N_2}Q_{L_2}u_2\|_{L^2}
$$
 leads to 
\begin{equation}\label{trivilabi}
\|(P_{N_1} u_1)(P_{N_2} u_2)\|_{L^2} 
\lesssim   (N_1 \wedge N_2)
\|P_{N_1} u_1\|_{X^{0,\frac{1}{4}+} }\,\|P_{N_2}u_2\|_{X^{0,\frac{1}{4}+}},
\end{equation}
Interpolating this last estimate with the following bilinear estimate obtained by 
Pilod and the first author (\textit{c.f.}  Proposition 3.6 in \cite{MP15}) 
\begin{equation}\label{bi11}
\|(P_{N_1} u_1)(P_{N_2} u_2)\|_{L^2}
\lesssim \frac{(N_1 \wedge N_2)^{\frac{1}{2}}}{N_1\lor N_2}
\|P_{N_1} u_1\|_{X^{0,\frac{1}{2}+} }\,\|P_{N_2}u_2\|_{X^{0,\frac{1}{2}+}},
\end{equation}
we obtain \eqref{bi1}.
\end{proof}

\subsection{Multilinear estimate for gZK in  $\R \times \T$}
With Propositions \ref{prop31}-\ref{prop32} in hand, we can prove the main result of  this section that  reads as follows.
\begin{proposition}\label{progZK} Let $ k\ge 2 $ and let $s> 5/6$ if $k=2 $ and $ s> 1- $ if 
 $k\ge 3$. Then, for $0<\delta\ll 1 $ it holds
\begin{equation}\label{ME-gZK-a}
\|\partial_x(\prod_{i=1}^{k+1} u_i)\|_{X^{s, -\frac{1}{2}+2\delta}}\le \sum_{\sigma\in {\mathcal S}_k}C_{k,s} 
\|u_{\sigma(1)}\|_{X^{s, \frac{1}{2}+\delta}}\prod_{i=2}^3 \|u_{\sigma(i)}\|_{X^{\frac{5}{6}+, \frac{1}{2}+\delta}}
\prod\limits_{j=4}^{k+1}\|u_{\sigma(j)}\|_{X^{1-, \frac{1}{2}+\delta}}.
\end{equation}
\end{proposition}
\begin{proof}
We use the non homogeneous decomposition to write 
$$
I_k:=\|\partial_x(\prod_{i=1}^{k+1} u_i)\|_{X^{s, -\frac{1}{2}+2\delta}} \lesssim 
\sum_{N\ge 1} \sum_{N_1, N_2, .., N_k\ge 1} N^{1+s}
 \Bigl\|P_N \Bigl( \prod_{i=1}^{k+1} P_{N_i} u_i\Bigr)\Bigr\|_{X^{0, -\frac{1}{2}+2\delta}}
$$
By symmetry  we may assume $ N_1\ge N_2 \ge .. \ge N_{k+1}$.\\
${\bf 1.} $ $ N\lesssim 1$. Then we have 
\begin{eqnarray*}
I_k  &\lesssim  &  \sum_{N_1\ge N_2\ge ..\ge N_{k+1}} 
 \Bigl\|P_{\lesssim 1} \Bigl( \prod_{i=1}^{k+1} P_{N_i} u_i\Bigr)\Bigr\|_{L^2_t L^1_{xy}} \\
 & \lesssim &   \sum_{N_1\ge N_2\ge ..\ge N_{k+1}} \prod_{i=1}^{k+1}  \|P_{N_i} u_i \|_{L^{2(k+1)}_t L^{k+1}_{xy}}   \\
 &\lesssim & \prod\limits_{i=1}^{k+1}\|u_i\|_{X^{(1-\frac{2}{k+1})+ , \frac{1}{2}+\delta}}.
\end{eqnarray*}
${\bf 2.}  $ $N \gg 1 $. We separate  two contributions.

${\bf 2.1.} $ $N_3\gtrsim N_1 $.  Then we must have $ N_3\gtrsim N/k $ and using \eqref{stri4} on $ u_1, u_2, u_3 $ and its dual estimate 
we can bound this contribution by 
\begin{eqnarray*}
I_k  &\lesssim  &   \sum_{N\gg 1} \sum_{N_1\ge N_2\ge N_3\gtrsim N/k}  \sum_{N_4,.., N_{k+1}\le N_3} 
 N^{1+s+1/6} \Bigl\|P_{N} \Bigl( \prod_{i=1}^{k+1} P_{N_i} u_i\Bigr)\Bigr\|_{L^{4/3}_{txy}} \\
 &\lesssim  &  \sum_{N\gg 1} \sum_{N_1\ge N_2\ge N_3\gtrsim N/k}  \sum_{N_4,.., N_{k+1}\le N_3} 
 N^{1+s+1/6} \prod_{i=1}^3 \|P_{N_i} u_i \|_{L^4_{txy}} 
 \Bigl\|\prod_{i=4}^{k+1} P_{N_i} u_i\Bigr\|_{L^{\infty}_{txy}} \\
 &\lesssim  &  \sum_{N\gg 1} \sum_{N_1\ge N_2\ge N_3\gtrsim N/k}  \sum_{N_4,.., N_{k+1}\le N_3} 
 N^{1+s+1/6} N_1^{-s+1/6}  (N_2 N_3)^{-\theta+1/6} \\
& & \hspace*{50mm} \|P_{N_1} u_1 \|_{X^{s,3/8}}\prod_{i=2}^3 \|P_{N_i} u_i \|_{X^{\theta,3/8}} 
\prod_{i=4}^{k+1} \Bigl\| P_{N_i} u_i\Bigr\|_{L^{\infty}_{txy}} \; ,
\end{eqnarray*}
for $ \theta>0 $.
We notice that  $ -2\theta +1 +2/3<0 \Leftrightarrow \theta>5/6  $ and thus that for $ \theta=5/6+\varepsilon  $ with $\varepsilon>0 $ it holds 
$$
N^{1+s+1/6} N_1^{-s+1/6}  (N_2 N_3)^{-\theta+1/6}\le c_k N_1^{-2\varepsilon} \; .
$$
On the other hand by Bernstein and Sobolev inequalities since $ N_1=\displaystyle\max_{i=1,..,k+1} N_i $ it holds 
\begin{equation}\label{tuy}
N_1^{0-} \prod_{k=4}^{k+1} \|P_{N_k} u_k \|_{L^\infty_{txy}} \lesssim N_1^{0-} \prod_{k=4}^{k+1} N_k \|P_{N_k} u_k \|_{L^\infty_{t} L^2_{xy}} 
\lesssim  \prod_{k=4}^{k+1} \|P_{N_k} u_k \|_{X^{1-,\frac{1}{2}+}} \; .
\end{equation}
 Therefore, using part of the coefficient  $N_1^{-2\varepsilon} $ for  \eqref{tuy} and another part to re-sum   we get 
$$
I_k    
 \lesssim  c_k  \|P_{N_1} u_1 \|_{X^{s,3/8}} \prod_{i=2}^3 \|P_{N_i} u_i \|_{X^{\frac{5}{6}+,3/8}} 
\prod_{i=4}^{k+1} \Bigl\| P_{N_i} u_i\Bigr\|_{X^{1-,1/2+}} 
$$
${\bf 2.2.} $ $ N_3\ll N_1 $.\\
${\bf 2.2.1.} $ $ N_2\gtrsim N $.
 Then by using the dual estimate of \eqref{stri4}  together with \eqref{stri4}, \eqref{bi11} and again \eqref{tuy}  we get 
\begin{eqnarray*}
I_k  &\lesssim  &  
 \sum_{N\gg 1\atop N_1\ge N_2 \gtrsim N} \sum_{N_{k+1}\le ..\le N_3 \ll N_1}   
  \hspace*{-8mm}  N^{1+s+1/6} \Bigl\|P_{N} \Bigl( \prod_{i=1}^{k+1} P_{N_i} u_i\Bigr)\Bigr\|_{L^{4/3}_{txy}} \\
 & \lesssim &  \sum_{N\gg 1\atop N_1\ge N_2 \gtrsim N} \sum_{N_{k+1}\le ..\le N_3 \ll N_1}   
 \hspace*{-8mm}  N^{1+s} N^{1/6} \|P_{N_1} u_1 P_{N_3} u_3 \|_{L^2_{txy}} 
\|P_{N_2} u_2 \prod_{i=4}^{k+1} P_{N_i} u_i  \|_{L^4_{txy}}\\
& \lesssim &\sum_{N\gg 1\atop N_1\ge N_2 \gtrsim N} \sum_{N_{k+1}\le ..\le N_3 \ll N_1}
 \hspace*{-8mm}  N^{1+s} N^{1/6} N_1^{-1} N_3^{1/2}\| P_{N_1} u_1 \|_{X^{0,\frac{1}{2}+}} \| P_{N_3}u_3 \|_{X^{0,\frac{1}{2}+}}  
\| P_{N_2} u_2\|_{L^{4}_{txy}}  \prod_{i=4}^{k+1} \|P_{N_i} u_i  \|_{L^{\infty}_{txy}}\\
 & \lesssim & \sum_{N\gg 1\atop N_1\ge N_2 \gtrsim N} \sum_{N_{k+1}\le ..\le N_3 \ll N_1}
 \hspace*{-8mm}  N^{1+s+1/6} N_1^{-1} N_3^{1/2} N_2^{\frac{1}{6}}  N_1^{0+}
 \prod_{i=1}^3 \|P_{N_i} u_i \|_{X^{0,\frac{1}{2}+}} 
\prod_{i=4}^{k+1} \Bigl\| P_{N_i} u_i\Bigr\|_{X^{1-,\frac{1}{2}+}} \\
& \lesssim &   \|u_1 \|_{X^{s,\frac{1}{2}+}} 
 \prod_{i=2}^3 \| u_2 \|_{X^{\frac{5}{12}+,\frac{1}{2}+}} 
\prod_{i=4}^{k+1} \Bigl\|  u_i\Bigr\|_{X^{1-,\frac{1}{2}+}} 
\end{eqnarray*}
since $ N^{1+s+1/6} N_1^{-1} N_3^{1/2} N_2^{\frac{1}{6}} N_1^{0+} \lesssim 
N_1^{-\varepsilon}  N_1^s N_2^{\frac{5}{12}+\varepsilon} N_3^{\frac{5}{12}+\varepsilon} $
 for $ 0<\varepsilon\ll 1$.\\
${\bf 2.2.2.}$   $ N_2 \ll N $. By duality it suffices to prove that for any $ v  \in X^{0, \frac{1}{2}-} $
 it holds 
 $$
 J_k:= N^{1+s} \Bigl| \int_{\R^3}  \Bigl( \prod_{i=1}^{k+1} P_{N_i} u_i\Bigr) P_N v \Bigr| \lesssim 
   N_1^{0-}    \|v\|_{X^{0,\frac{1}{2}-}}\|u_1 \|_{X^{s,\frac{1}{2}+}} 
 \prod_{i=2}^3 \| u_2 \|_{X^{0+,\frac{1}{2}+}} 
\prod_{i=4}^{k+1} \Bigl\|  u_i\Bigr\|_{X^{1-,\frac{1}{2}+}} 
$$
We write 
$$
J_k \lesssim N^{1+s}  \| P_{N_1} u_1 P_{N_3} u_3\|_{L^2}  \| P_{N} v P_{N_2} u_2\|_{L^2}
\| P_{\lesssim k N_1}\Bigl( \prod_{i=4}^{k+1} P_{N_i} u_i\Bigr) \|_{L^\infty_{txy}}
$$
and  use two times \eqref{bi1} to get 
\begin{eqnarray*}
J_k &\lesssim &  N^{1+s} N_1^{-1+} N^{-1+} N_2^{\frac{1}{2}+} N_3^{\frac{1}{2}+} 
\|P_N v\|_{X^{0,\frac{1}{2}-}} \prod_{i=1}^3 \|P_{N_i}  u_i\|_{X^{0,\frac{1}{2}-}} 
\|P_{\lesssim k N_1} \Bigl( \prod_{i=4}^{k+1} P_{N_i} u_i \Bigr)\|_{L^\infty_{txy}}\\
& \lesssim & N_1^{0-}    \|v\|_{X^{0,\frac{1}{2}-}}\|u_1 \|_{X^{s,\frac{1}{2}+}} 
 \prod_{i=2}^3 \| u_i \|_{X^{0+,\frac{1}{2}+}} 
  N_1^{0-}\| P_{\lesssim k N_1} \Bigl( \prod_{i=4}^{k+1} P_{N_i} u_i \Bigr) \|_{L^\infty_{txy}}\\
&  \lesssim &
   N_1^{0-}    \|v\|_{X^{0,\frac{1}{2}-}}\|u_1 \|_{X^{s,\frac{1}{2}+}} 
 \prod_{i=2}^3 \| u_2 \|_{X^{0+,\frac{1}{2}+}} 
\prod_{i=4}^{k+1} \Bigl\|  u_i\Bigr\|_{X^{1-,\frac{1}{2}+}} 
\end{eqnarray*}
that proves the desired result.
\end{proof}
\subsection{Proof of the LWP}
Let us now recall the  following well-known estimates for
Bourgain's spaces (see for instance Ginibre \cite{G96})

\begin{lemma}[Homogeneous linear estimate] \label{prop1.1}
Let $s \in \mathbb R$ and $b>\frac12$. Then
\begin{equation} \label{prop1.1.2}
\|\eta(t) e^{-t\partial_x\Delta} f\|_{X^{s,b}} \lesssim\|f\|_{H^s} \ .
\end{equation}
\end{lemma}

\begin{lemma}[Non-homogeneous linear estimate] \label{prop1.2}
Let $s \in \mathbb R$. Then for any $0<\delta<\frac12$,
\begin{equation} \label{prop1.2.1}
\big\|\eta(t)\int_0^t e^{-(t-t')\partial_x\Delta}g(t')dt'\big\|_{X^{s,\frac12+\delta}} \lesssim  \|g\|_{X^{s,-\frac12+\delta}} \ .
\end{equation}
\end{lemma}

\begin{lemma} \label{prop1.3b}
For any $T>0$, $s \in \mathbb R$ and for all $-\frac12< b' \le b
<\frac12$, it holds
\begin{equation} \label{prop1.3b.1}
\|u\|_{X^{s,b'}_T} \lesssim T^{b-b'}\|u\|_{X^{s,b}_T}.
\end{equation}
\end{lemma}
Combining Proposition \ref{progZK} with the above lemmas we obtain
   that for any $ 0<T<1 $ small enough, the functional  
  $$
  {\mathcal G}_T(w)(t,\cdot):=\psi(t)  e^{-t\partial_x\Delta} u_0 -\frac{1}{2} \int_0^t e^{-(t-t')\partial_x\Delta} \partial_x (\psi(\cdot/T) w)^{k+1}(t',\cdot)\, dt' 
  $$
  maps $X^{s,\frac{1}{2}+} $ into itself and is strictly contractive. This yields to the local well-posedness result and we refer the reader to \cite{ST01} for the details. 


\section{Global well-posedness in the energy space}

The global existence in $ H^1(\R\times\T) $ of the solutions to \eqref{ZKk} for small enough initial data is a straightforward consequence of the conservation laws \eqref{M}-\eqref{H} and a suitable Gagliardo-Niremberg type inequality. Next, we prove Lemma \ref{LemGN} and use it to deduce the global well-posedness result stated in Theorem \ref{global1}.

\subsection{Gagliardo Nirenberg type inequality}

\begin{proof}[Proof of Lemma \ref{LemGN}]
We follow the ideas of Hebey and Vaugon \cite{HV95}. Let $\eta_1, \eta_2 \in C^{\infty}(\mathbb{T})$ non negative functions such that
$$
\eta^2_1+ \eta^2_2=1, \quad \supp\,\eta_1\cap[0,1]\subset [1/4,3/4] \quad \,\mbox{ and }\,\quad \eta_1\equiv 1 \,\mbox{ on }\, [1/3,2/3].
$$
We have
\begin{equation}\label{feta}
\|f\|^{k+2}_{k+2}=\left\|f^2\right\|^{(k+2)/2}_{(k+2)/2}=\left\|\sum_{j=1}^2 \eta^2_j  f^2\right\|^{(k+2)/2}_{(k+2)/2}\leq \left(\sum_{j=1}^2 \left\| \eta_j  f\right\|^2_{k+2}\right)^{^{(k+2)/2}}.
\end{equation}
Since $\supp\,\eta_1\cap[0,1]\subset [1/4,3/4]$ and $\supp\,\eta_2\cap[1/2,3/2]\subset [2/3,4/3]$, we deduce for all $V\in H^1(\mathbb{R}\times\mathbb{T})$ that 
$$
\left\|\eta_1(y)V\right \|_{L^{k+2}(\mathbb{R}\times\mathbb{T})}^{k+2}=\left\|\eta_1(y)V\right \|_{L^{k+2}(\mathbb{R}\times(0,1))}^{k+2}\leq C_{k,\mathbb{R}}\left\|\eta_1(y)V\right\|^2_2 \left\|\nabla (\eta_1(y)V)\right\|_2^{k}
$$
and an analogous estimate for $\eta_2$.

Inserting the previous inequality in \eqref{feta} leads to
$$
\|f\|^{k+2}_{k+2}\leq C_{k,\mathbb{R}} \left(\sum_{j=1}^2 \left\|\eta_j(y)f\right\|^{4/(k+2)}_2 \left\|\nabla (\eta_j(y)f)\right\|_2^{2k/(k+2)}\right)^{^{(k+2)/2}}.
$$
Now, observe that
$$
\left\|\eta_j(y)f\right\|^{4/(k+2)}_2=\left( \int \eta^2_j f^2 dxdy\right)^{2/(k+2)}
$$
and
\begin{eqnarray}
\left\|\nabla (\eta_j(y)f)\right\|_2^{2k/(k+2)}&=&\left(\int\left(\eta^2_j|\nabla f|^2+2\eta_j\partial_y \eta_jf\partial_yf+(\partial_y\eta_j)^2f^2\right)dxdy\right)^{k/(k+2)}\\
&=&\left(\int\left(\eta^2_j|\nabla f|^2+\frac{1}{2}\partial_y(\eta^2_j) \partial_y(f^2)+(\partial_y\eta_j)^2f^2\right)dxdy\right)^{k/(k+2)}\\
&=&\left(\int\left(\eta^2_j|\nabla f|^2+H_jf^2\right)dxdy\right)^{k/(k+2)},
\end{eqnarray}
where we have used integration by parts in the last step and $H_j=-\frac{1}{2}\partial^2_y(\eta^2_j) +(\partial_y\eta_j)^2$.


Collecting the last two identities we obtain
$$
\|f\|^{k+2}_{k+2}\leq C_{k,\mathbb{R}} \left(\sum_{j=1}^2  \left( \int \eta^2_j f^2 dxdy\right)^{2/(k+2)} \left(\int\left(\eta^2_j|\nabla f|^2+H_jf^2\right)dxdy\right)^{k/(k+2)}\right)^{^{(k+2)/2}}.
$$
Using, by Holder inequality, that for any couples $(a_1,b_1)$ and $(a_2,b_2)$ in $\mathbb{R}_+\times \mathbb{R}_+$
$$
\sum_{j=1}^2a_j^{2/(k+2)}b_j^{k/(k+2)}\leq \left(\sum_{j=1}^2a_j\right)^{^{2/(k+2)}}\left(\sum_{j=1}^2b_j\right)^{^{k/(k+2)}}
$$
we finally get
\begin{eqnarray}
\|f\|^{k+2}_{k+2}&\leq & C_{k,\mathbb{R}}\left( \sum_{j=1}^2 \int \eta^2_j f^2 dxdy\right)\left(\int\sum_{j=1}^2\left(\eta^2_j|\nabla f|^2+H_jf^2\right)dxdy\right)^{k/2}\\
&=&C_{k,\mathbb{R}}\|f\|^2_2 \left(\|\nabla f\|_2^2+\int\sum_{j=1}^2 H_jf^2dxdy\right)^{k/2}\\
&\leq &C_{k,\mathbb{R}}\|f\|^2_2 \left(\|\nabla f\|_2^2+c\|f\|^2_2\right)^{k/2},
\end{eqnarray}
where we have used that $H_j \in L^{\infty}$ in the last step, which completes the proof.
\end{proof}

\subsection{Proof of the GWP}

We are now in position to prove our global well-posedness result. 

\begin{proof}[Proof of Theorem \ref{global1}]
From the sharp Gagliardo Nirenberg type inequality \eqref{SGN} with $k=2$, we have
$$
\|\nabla u(t)\|_2^2\leq 2H(u_0)+\frac{C_{2,\mathbb{R}}}{2}\|u_0\|_2^2\left(\|\nabla u(t)\|_2^2+C_{2,\mathbb{T}}\|u_0\|_2^2\right)
$$
and then, from \eqref{opt1}, we get
$$
\|\nabla u(t)\|_2^2\left(1-\frac{\|u_0\|^2_2}{\|Q_2\|^2_2}\right)\leq 2H(u_0)+\frac{C_{2,\mathbb{T}}\|u_0\|_2^4}{\|Q_2\|^2_2}.
$$
So, assuming \eqref{GR0}, we deduce that $\|\nabla u(t)\|_2$ is uniformly bounded in time, which implies the desired result in $(i)$ in view of the local theory Theorem \ref{LWP}.

Next, we turn our attention to second part of the theorem. Again, the sharp Gagliardo Nirenberg type inequality \eqref{SGN}  yields
$$
\|\nabla u(t)\|_2^2\leq 2H(u_0)+\frac{2C_{k,\mathbb{R}}}{k+2}\|u_0\|_2^2\left(\|\nabla u(t)\|_2^2+C_{k,\mathbb{T}}\|u_0\|_2^2\right)^{k/2}
$$
Defining $X(t)=\|\nabla u(t)\|_2^2+C_{k,\mathbb{T}}\|u_0\|_{2}^2$, $A_k=2H(u_0)+C_{k,\mathbb{T}}\|u_0\|_{2}^2$ and $B_k=2C_{k,\mathbb{R}}\|u_0\|_2^2/(k+2)$, the above inequality implies for every existence time $t$ that
$$
X(t)-B_kX(t)^{k/2}\leq A_k.
$$
Let $f(x)=x-B_kx^{k/2}$, it is easy to see that this function has an absolute maximum at $x_0=(2/kB_k)^{2/(k-2)}$ with maximum value $f(x_0)=(2/kB_k)^{2/(k-2)}(k-2)/2$. A continuity argument guarantees that, if 
\begin{equation}\label{relglobal}
A_k<f(x_0) \quad \mbox{and} \quad X(0)<x_0,
\end{equation}
then $X(t)<x_0$ for all existence time $t>0$ and thus the solution is global. From identities \eqref{opt1} and  \eqref{rel2}, we deduce that \eqref{GR1} and \eqref{GR2} imply the above inequalities, completing the proof.
\end{proof}

\vspace{0.5cm}
\noindent 
\textbf{Acknowledgments.} L.M. would like to thank Emmanuel Humbert  for his explanations on  the proof of the main result  in \cite{HV95}. L.G.F. was partially supported by Coordena\c{c}\~ao de Aperfei\c{c}oamento de Pessoal de N\'ivel Superior - CAPES, Conselho Nacional de Desenvolvimento Cient\'ifico e Tecnol\'ogico - CNPq and Funda\c{c}\~ao de Amparo a Pesquisa do Estado de Minas Gerais - Fapemig/Brazil. 


\begin{thebibliography}{38}
\providecommand{\natexlab}[1]{#1}
\providecommand{\url}[1]{\texttt{#1}}
\expandafter\ifx\csname urlstyle\endcsname\relax
  \providecommand{\doi}[1]{doi: #1}\else
  \providecommand{\doi}{doi: \begingroup \urlstyle{rm}\Url}\fi

\bibitem[Angulo et~al.(2002)Angulo, Bona, Linares, and Scialom]{ABLS02}
J.~Angulo, J.~L. Bona, F.~Linares, and M.~Scialom.
\newblock Scaling, stability and singularities for nonlinear, dispersive wave
  equations: the critical case.
\newblock \emph{Nonlinearity}, 15\penalty0 (3):\penalty0 759--786, 2002.

\bibitem[Biagioni and Linares(2003)]{BL01}
H.~A. Biagioni and F.~Linares.
\newblock Well-posedness results for the modified {Z}akharov-{K}uznetsov
  equation.
\newblock In \emph{Nonlinear equations: methods, models and applications
  ({B}ergamo, 2001)}, volume~54 of \emph{Progr. Nonlinear Differential
  Equations Appl.}, pages 181--189. Birkh\"{a}user, Basel, 2003.

\bibitem[Bourgain(1993)]{B93II}
J.~Bourgain.
\newblock Fourier transform restriction phenomena for certain lattice subsets
  and applications to nonlinear evolution equations. {II}. {T}he
  {K}d{V}-equation.
\newblock \emph{Geom. Funct. Anal.}, 3\penalty0 (3):\penalty0 209--262, 1993.

\bibitem[Bridges(2000)]{Br00}
T.~J. Bridges.
\newblock Universal geometric condition for the transverse instability of
  solitary waves.
\newblock \emph{Phys. Rev. Lett.}, 84:\penalty0 2614--2617, Mar 2000.

\bibitem[C\^{o}te et~al.(2016)C\^{o}te, Mu\~{n}oz, Pilod, and Simpson]{CMP16}
R.~C\^{o}te, C.~Mu\~{n}oz, D.~Pilod, and G.~Simpson.
\newblock Asymptotic stability of high-dimensional {Z}akharov-{K}uznetsov
  solitons.
\newblock \emph{Arch. Ration. Mech. Anal.}, 220\penalty0 (2):\penalty0
  639--710, 2016.

\bibitem[de~Bouard(1996)]{Bou96}
A.~de~Bouard.
\newblock Stability and instability of some nonlinear dispersive solitary waves
  in higher dimension.
\newblock \emph{Proc. Roy. Soc. Edinburgh Sect. A}, 126\penalty0 (1):\penalty0
  89--112, 1996.

\bibitem[Duyckaerts et~al.(2008)Duyckaerts, Holmer, and Roudenko]{DHR08}
T.~Duyckaerts, J.~Holmer, and S.~Roudenko.
\newblock Scattering for the non-radial 3{D} cubic nonlinear {S}chr\"odinger
  equation.
\newblock \emph{Math. Res. Lett.}, 15\penalty0 (6):\penalty0 1233--1250, 2008.

\bibitem[Faminski\u{\i}(1995)]{Fa95}
A.~V. Faminski\u{\i}.
\newblock The {C}auchy problem for the {Z}akharov-{K}uznetsov equation.
\newblock \emph{Differential Equations}, 31\penalty0 (6):\penalty0 1002--1012,
  1995.

\bibitem[Farah et~al.(2011)Farah, Linares, and Pastor]{FLP11}
L.~G. Farah, F.~Linares, and A.~Pastor.
\newblock The supercritical generalized {K}d{V} equation: global well-posedness
  in the energy space and below.
\newblock \emph{Math. Res. Lett.}, 18\penalty0 (2):\penalty0 357--377, 2011.

\bibitem[Farah et~al.(2012)Farah, Linares, and Pastor]{FLP12}
L.~G. Farah, F.~Linares, and A.~Pastor.
\newblock A note on the 2{D} generalized {Z}akharov-{K}uznetsov equation:
  local, global, and scattering results.
\newblock \emph{J. Differential Equations}, 253\penalty0 (8):\penalty0
  2558--2571, 2012.

\bibitem[Farah et~al.(2014)Farah, Linares, and Pastor]{FLP14}
L.~G. Farah, F.~Linares, and A.~Pastor.
\newblock Global well-posedness for the {$k$}-dispersion generalized
  {B}enjamin-{O}no equation.
\newblock \emph{Differential Integral Equations}, 27\penalty0 (7-8):\penalty0
  601--612, 2014.

\bibitem[Farah et~al.(2019{\natexlab{a}})Farah, Holmer, and Roudenko]{FHR19}
L.~G. Farah, J.~Holmer, and S.~Roudenko.
\newblock Instability of solitons in the 2d cubic {Z}akharov-{K}uznetsov
  equation.
\newblock In \emph{Nonlinear dispersive partial differential equations and
  inverse scattering}, volume~83 of \emph{Fields Inst. Commun.}, pages
  295--371. Springer, New York, 2019{\natexlab{a}}.

\bibitem[Farah et~al.(2019{\natexlab{b}})Farah, Holmer, and Roudenko]{FHR19II}
L.~G. Farah, J.~Holmer, and S.~Roudenko.
\newblock Instability of solitons---revisited, {II}: {T}he supercritical
  {Z}akharov-{K}uznetsov equation.
\newblock In \emph{Nonlinear dispersive waves and fluids}, volume 725 of
  \emph{Contemp. Math.}, pages 89--109. Amer. Math. Soc., [Providence], RI,
  2019{\natexlab{b}}.

\bibitem[Ginibre(1996)]{G96}
J.~Ginibre.
\newblock Le probl\`eme de {C}auchy pour des {EDP} semi-lin\'{e}aires
  p\'{e}riodiques en variables d'espace (d'apr\`es {B}ourgain).
\newblock Number 237, pages Exp. No. 796, 4, 163--187. 1996.
\newblock S\'{e}minaire Bourbaki, Vol. 1994/95.

\bibitem[Gr\"unrock(2015)]{G15}
A.~Gr\"unrock.
\newblock On the generalized {Z}akharov-{K}uznetsov equation at critical
  regularity.
\newblock \emph{ArXiv preprint arXiv:1509.09146}, 2015.

\bibitem[Gr\"{u}nrock and Herr(2014)]{GH14}
A.~Gr\"{u}nrock and S.~Herr.
\newblock The {F}ourier restriction norm method for the {Z}akharov-{K}uznetsov
  equation.
\newblock \emph{Discrete Contin. Dyn. Syst.}, 34\penalty0 (5):\penalty0
  2061--2068, 2014.

\bibitem[Hebey and Vaugon(1995)]{HV95}
E.~Hebey and M.~Vaugon.
\newblock The best constant problem in the {S}obolev embedding theorem for
  complete {R}iemannian manifolds.
\newblock \emph{Duke Math. J.}, 79\penalty0 (1):\penalty0 235--279, 1995.

\bibitem[Holmer and Roudenko(2008)]{HR08}
J.~Holmer and S.~Roudenko.
\newblock A sharp condition for scattering of the radial 3{D} cubic nonlinear
  {S}chr\"{o}dinger equation.
\newblock \emph{Comm. Math. Phys.}, 282\penalty0 (2):\penalty0 435--467, 2008.

\bibitem[Kenig and Merle(2006)]{KM06}
C.~E. Kenig and F.~Merle.
\newblock Global well-posedness, scattering and blow-up for the
  energy-critical, focusing, non-linear {S}chr\"odinger equation in the radial
  case.
\newblock \emph{Invent. Math.}, 166\penalty0 (3):\penalty0 645--675, 2006.

\bibitem[Kinoshita(2021)]{K21}
S.~Kinoshita.
\newblock Global well-posedness for the {C}auchy problem of the
  {Z}akharov-{K}uznetsov equation in 2{D}.
\newblock \emph{Ann. Inst. H. Poincar\'{e} C Anal. Non Lin\'{e}aire},
  38\penalty0 (2):\penalty0 451--505, 2021.

\bibitem[Kinoshita(2022)]{K22}
S.~Kinoshita.
\newblock Well-posedness for the {C}auchy problem of the modified
  {Z}akharov-{K}uznetsov equation.
\newblock \emph{Funkcial. Ekvac.}, 65\penalty0 (2):\penalty0 139--158, 2022.

\bibitem[Kuznetsov and Zakharov(2074)]{ZK74}
E.~A. Kuznetsov and V.~E. Zakharov.
\newblock On three dimensional solitons.
\newblock \emph{Sov. Phys. JETP.}, 39:\penalty0 285--286, 2074.

\bibitem[Lannes et~al.(2013)Lannes, Linares, and Saut]{LLS13}
D.~Lannes, F.~Linares, and J.-C. Saut.
\newblock The {C}auchy problem for the {E}uler-{P}oisson system and derivation
  of the {Z}akharov-{K}uznetsov equation.
\newblock In \emph{Studies in phase space analysis with applications to
  {PDE}s}, volume~84 of \emph{Progr. Nonlinear Differential Equations Appl.},
  pages 181--213. Birkh\"{a}user/Springer, New York, 2013.

\bibitem[Linares and Pastor(2009)]{LP09}
F.~Linares and A.~Pastor.
\newblock Well-posedness for the two-dimensional modified
  {Z}akharov-{K}uznetsov equation.
\newblock \emph{SIAM J. Math. Anal.}, 41\penalty0 (4):\penalty0 1323--1339,
  2009.

\bibitem[Linares and Pastor(2011)]{LP11}
F.~Linares and A.~Pastor.
\newblock Local and global well-posedness for the 2{D} generalized
  {Z}akharov-{K}uznetsov equation.
\newblock \emph{J. Funct. Anal.}, 260\penalty0 (4):\penalty0 1060--1085, 2011.

\bibitem[Luo(2022)]{LY21}
Y.~Luo.
\newblock Large data global well-posedness and scattering for the focusing
  cubic nonlinear {S}chr\"{o}dinger equation on $\mathbb{R}^2\times
  \mathbb{T}$.
\newblock \emph{ArXiv preprint arXiv:2202.10219}, 2022.

\bibitem[Martel and Merle(2001)]{MM01}
Y.~Martel and F.~Merle.
\newblock Asymptotic stability of solitons for subcritical generalized {K}d{V}
  equations.
\newblock \emph{Arch. Ration. Mech. Anal.}, 157\penalty0 (3):\penalty0
  219--254, 2001.

\bibitem[Martel and Merle(2005)]{MM05}
Y.~Martel and F.~Merle.
\newblock Asymptotic stability of solitons of the subcritical g{K}d{V}
  equations revisited.
\newblock \emph{Nonlinearity}, 18\penalty0 (1):\penalty0 55--80, 2005.

\bibitem[Martel and Merle(2008)]{MM08}
Y.~Martel and F.~Merle.
\newblock Asymptotic stability of solitons of the g{K}d{V} equations with
  general nonlinearity.
\newblock \emph{Math. Ann.}, 341\penalty0 (2):\penalty0 391--427, 2008.

\bibitem[Molinet and Pilod(2015)]{MP15}
L.~Molinet and D.~Pilod.
\newblock Bilinear {S}trichartz estimates for the {Z}akharov-{K}uznetsov
  equation and applications.
\newblock \emph{Ann. Inst. H. Poincar\'{e} C Anal. Non Lin\'{e}aire},
  32\penalty0 (2):\penalty0 347--371, 2015.

\bibitem[Pego and Weinstein(1992)]{PW92}
R.~L. Pego and M.~I. Weinstein.
\newblock Eigenvalues, and instabilities of solitary waves.
\newblock \emph{Philos. Trans. Roy. Soc. London Ser. A}, 340\penalty0
  (1656):\penalty0 47--94, 1992.

\bibitem[Ribaud and Vento(2012)]{RV12}
F.~Ribaud and S.~Vento.
\newblock A note on the {C}auchy problem for the 2{D} generalized
  {Z}akharov-{K}uznetsov equations.
\newblock \emph{C. R. Math. Acad. Sci. Paris}, 350\penalty0 (9-10):\penalty0
  499--503, 2012.

\bibitem[Rousset and Tzvetkov(2008)]{RT08}
F.~Rousset and N.~Tzvetkov.
\newblock Transverse nonlinear instability of solitary waves for some
  {H}amiltonian {PDE}'s.
\newblock \emph{J. Math. Pures Appl. (9)}, 90\penalty0 (6):\penalty0 550--590,
  2008.

\bibitem[Rousset and Tzvetkov(2009)]{RT09}
F.~Rousset and N.~Tzvetkov.
\newblock Transverse nonlinear instability for two-dimensional dispersive
  models.
\newblock \emph{Ann. Inst. H. Poincar\'{e} C Anal. Non Lin\'{e}aire},
  26\penalty0 (2):\penalty0 477--496, 2009.

\bibitem[Saut and Tzvetkov(2001)]{ST01}
J.-C. Saut and N.~Tzvetkov.
\newblock On periodic {KP}-{I} type equations.
\newblock \emph{Comm. Math. Phys.}, 221\penalty0 (3):\penalty0 451--476, 2001.

\bibitem[Weinstein(1982/83)]{W83}
M.~I. Weinstein.
\newblock Nonlinear {S}chr\"{o}dinger equations and sharp interpolation
  estimates.
\newblock \emph{Comm. Math. Phys.}, 87\penalty0 (4):\penalty0 567--576,
  1982/83.

\bibitem[Yamazaki(2017)]{YY17}
Y.~Yamazaki.
\newblock Stability for line solitary waves of {Z}akharov-{K}uznetsov equation.
\newblock \emph{J. Differential Equations}, 262\penalty0 (8):\penalty0
  4336--4389, 2017.

\bibitem[Yu et~al.(2021)Yu, Yue, and Zhao]{YYZ21}
X.~Yu, H.~Yue, and Z.~Zhao.
\newblock Global {W}ell-posedness for the focusing cubic {NLS} on the product
  space {$\Bbb{R} \times \Bbb{T}^3$}.
\newblock \emph{SIAM J. Math. Anal.}, 53\penalty0 (2):\penalty0 2243--2274,
  2021.

\end{thebibliography}

\newcommand{\Addresses}{{
		\bigskip
		\footnotesize
		
		LUC MOLINET, \textsc{Universit\'e de Tours, Universit\'e d'Orl\'eans,
CNRS, France}\par\nopagebreak
		\textit{E-mail address:} \texttt{Luc.Molinet@lmpt.univ-tours.fr}
		
		\medskip
		
		LUIZ G. FARAH, \textsc{Department of Mathematics, UFMG, Brazil}\par\nopagebreak
		\textit{E-mail address:} \texttt{farah@mat.ufmg.br}

}}
\setlength{\parskip}{0pt}
\Addresses

\end{document}